\documentclass[a4paper]{amsart}

\usepackage{xcolor}

\def\AA{{\mathbb{A}}}

\def\RR{{\mathbb{R}}}
\def\CC{{\mathbb{C}}}
\def\QQ{{\mathbb{Q}}}
\def\NN{{\mathbb{N}}}
\def\ZZ{{\mathbb{Z}}}

\let \cedilla =\c

\renewcommand{\b}{{\mathfrak{b}}}
\newcommand{\tb}{{\tilde{\b}}}

\renewcommand{\a}{{\mathfrak{a}}}
\newcommand{\p}{{\mathfrak{p}}}
\renewcommand{\wp}{{\widetilde{\p}}}

\newcommand{\oo}{{\mathcal O}}
\newcommand{\mld}{{\mathrm{mld}}}

\newcommand{\cont}{{\mathrm{Cont}}}
\newcommand{\codim}{{\mathrm{codim}}}
\newcommand{\ord}{{\mathrm{ord}}}
\newcommand{\mult}{{\mathrm{mult}}}
\newcommand{\spec}{{\mathrm{Spec \ }}}
\newcommand{\proj}{{\mathrm{Proj\ }}}

\newcommand{\wI}{{\widetilde{I}}}

\newcommand{\cha}{{\operatorname{char}}}

\newcommand{\m}{{\mathfrak{m}}}

\newcommand{\height}{{\mathrm{ht} }}

\newcommand{\lct}{{\operatorname{lct}}}
\renewcommand{\th}{{\tilde{h}}}

\renewcommand{\wr}{{\widetilde R}}
\newcommand{\tf}{{\tilde f}}
\newcommand{\we}{{\widetilde E}}

\newcommand{\wB}{{\widetilde{B}}}
\newcommand{\tg}{{\tilde g}}
\newcommand{\taa}{{\widetilde{\mathfrak a}}}
\newcommand{\ta}{{\widetilde{a}}}

\newcommand{\wa}{{\widetilde{A}}}
\newcommand{\ws}{{\widetilde{S}}}
\newcommand{\wsig}{{\widetilde{\Sigma}}}

\newcommand{\tP}{{\widetilde{P}}}

\newcommand{\wv}{{\widetilde{\varphi}}}

\newcommand{\tze}{{\tilde 0}}

 \newcommand{\wm}{{\widetilde\m}}

\newcommand{\zp}{{\ZZ/(p)}}

\renewcommand{\mod}{{\operatorname{mod}\ }}

\newcommand{\wC}{{\widetilde{C}}}
\renewcommand{\L}{{\mathcal{L}}}
\newcommand{\wL}{{\widetilde{\L}}}

\newtheorem{thm}{Theorem}[section]
\newtheorem{cor}[thm]{Corollary}
\newtheorem{Corollary-Definition}[thm]{Corollary-Definition}
\newtheorem{prop-def}[thm]{Proposition-Definition}
\newtheorem{prop}[thm]{Proposition}
\newtheorem{lem}[thm]{Lemma}

\theoremstyle{definition}
\newtheorem{defn}[thm]{Definition}
\newtheorem{ex}[thm]{Example}
\newtheorem{rem}[thm]{Remark}

\newtheorem{Proposition-Definition}[thm]{Proposition-Definition}

\begin{document}

\title[Liftings of ideals]{Liftings of ideals in positive characteristic to those in characteristic zero: \\
Surface case}

\thanks{
The author is partially supported by Grant-In-Aid (c) 22K03428 of JSPS 
and  by the Research Institute for Mathematical Sciences, an International Joint Usage/Research Center located in Kyoto University.
 }

\author{Shihoko Ishii }

\begin{abstract}
In this paper, we introduce the notion of a characteristic-zero lifting of an object in positive characteristic by means of ``skeletons''.
 Using this notion, we relate invariants of singularities in positive characteristic to their counterparts in characteristic zero. As an application, we prove that the set of log discrepancies for pairs consisting of a smooth surface and a multi-ideal is discrete. 
 We also show that the set of minimal log discrepancies and the set of log canonical thresholds of such pairs in positive characteristic are contained in the corresponding sets in characteristic zero.
 Another application is the construction of Campillo's complex model of a plane curve in positive characteristic via the skeleton lifting method.
\end{abstract}

\maketitle
\section{Introduction}
\noindent
The aim of this paper is to compare the invariants of pairs consisting of smooth varieties and multi-ideals in positive characteristic with their counterparts in characteristic zero. To this end, we introduce a method of lifting by means of ``skeletons.''

Our focus is on invariants of singularities, specifically log discrepancies, minimal log discrepancies, and log canonical thresholds.

In this paper, we prove a statement relating the positive-characteristic setting to its characteristic-zero counterpart. In its present form, however, this statement applies only to pairs consisting of a smooth surface and a multi-ideal. In dimension three or higher, there is a counterexample due to Koll\'ar (\cite{ko}); accordingly, a modified statement will be given in a forthcoming paper.

The main theorem of this paper serves as a bridge between characteristic \(0\) and positive characteristic.


\begin{thm}\label{newmain}
Let $k$ be an algebraically closed field of characteristic $p>0$, and set
\[
A_0=\AA_k^N.
\]
Let 
\[
A_n \xrightarrow{\varphi_n} A_{n-1} \xrightarrow{\varphi_{n-1}} \cdots
\xrightarrow{\varphi_2} A_1 \xrightarrow{\varphi_1} A_0
\]
be a sequence of blow-ups at closed points $P_i\in A_{i-1}$ dominating the origin $0\in A_0$.

Then there exists a sequence of varieties over $\CC$,
\[
\wa_n \xrightarrow{\widetilde\varphi_n} \wa_{n-1}
\xrightarrow{\widetilde\varphi_{n-1}} \cdots
\xrightarrow{\widetilde\varphi_2} \wa_1
\xrightarrow{\widetilde\varphi_1} \wa_0:=\AA^N_{\CC},
\]
such that:

\begin{enumerate}
\item [(i)]  $\wa_i \pmod p = A_i$ for all $0\le i\le n$;
\item[(ii)] for each $1\le i\le n$, the morphism
\[
\widetilde\varphi_i:\wa_i\to \wa_{i-1}
\]
is the blow-up at a closed point $\tP_i\in \wa_{i-1}$ dominating the origin of $\wa_0$;
\item[(iii)] if $E_i\subset A_i$ and $\we_i\subset \wa_i$ denote the
exceptional divisors of $\varphi_i$ and $\widetilde\varphi_i$, respectively, then
\[
k_{\we_i}=k_{E_i}
\qquad\text{for all }1\le i\le n;
\]
\item[(iv)] when $N=2$, for every $f\in k[x_1,x_2]$, there exists a lifting $\tf\in \CC[x_1,x_2]$
such that for every $i$ the following holds:
$$v_{E_i}(f)=v_{\we_i}(\tf).$$
\end{enumerate}
\end{thm}

\begin{cor}\label{main} 
Let  $k$ be an algebraically closed field of characteristic $p>0$.
    Let $E$ be a prime divisor over $A$ with the center at a smooth closed point $0$, and    $\a, \a_1,\ldots, \a_r\subset k[x_1,x_2]$ non-zero  ideals.
     Then, there exists  a prime divisor $\we$ over the affine plane $\wa=\AA_\CC^2$  over $\CC$ with the center
      at the origin $0$ satisfying the following:
  \begin{enumerate}
      \item[(i)]  $k_E=k_{\we}$;
      \item[(ii)]  there exists a lifting $\taa\subset \CC[x_1,x_2]$ of $\a$  satisfying
      $$v_E(\a)=v_{\we}(\taa);$$
     \item[(iii)] there exist liftings $\taa_i\subset \CC[x_1,x_2]$ of $\a_i$ $(i=1,\ldots,r)$ such that
     $$v_E(\a_i)=v_{\we}(\taa_i)$$
                and  for every multi-ideal
                 $\a^e:=\a_1^{e_1}\cdots\a_r^{e_r}$ $(e_i\in \RR_{>0})$ on $A$,
                 the lifting 
                 $\taa^e=\taa_1^{e_1}\cdots\taa_r^{e_r}$ of $\a^e$ on $\wa$
                 satisfies 
                 $$a(E;A,\a^e) =a(\we;\wa,\taa^e).$$ 
                    \end{enumerate}
\end{cor}

We obtain the following corollary as  applications of the Corollary \ref{main}.
 
   \begin{cor}[Discreteness of the set of log discrepancies]\label{equal-p-0}   Let  $k$ be an algebraically closed field of 
       characteristic $p>0$.
       Let $A$ be an algebraic surface over $k$ with a smooth closed point $0$. 
              Then, 
            for a fixed exponent $e\in \RR_{>0}^r$ and a real number $a>0$, the set of log discrepancies 
         $$\Lambda_{A,e,a}:=\left\{a(E;A,\a^e) \leq a \ \middle | \begin{array}{l}\ (A, \a^e)\ \mbox{is\ log\ canonical\ at $0$}, 
         \\ E:\ \mbox{\ any \ prime divisor\ over } (A,0) \end{array}\right\} $$
         is a finite set and does not depend on the choice of $(A,0)$.   
         In particular, for a  fixed  $e$, the set 
         $$\{\mld(0,A, \a^e) \mid \a^e\ \mbox{non-zero\ multi-ideal}\ \}$$ 
         is finite.
         Also for a fixed $\a^e$, the discreteness of log discrepancies gives another proof of  the existence of a prime divisor computing $\mld(0,A, \a^e)$.
    \end{cor}

\begin{cor}\label{mld-contain}
  The theorem also yields the following inclusions:
\[
\begin{aligned}
(1)\quad
&\Bigl\{
   \mld(0;A,\mathfrak a^e)
   \;\Bigm|\;
   \begin{array}[t]{@{}l@{}}
     (A,0)\text{ is a smooth point on a surface over }k,\\
     \mathfrak a^e\text{ is a multi-ideal}
   \end{array}
 \Bigr\} \\
&\qquad\subset
 \Bigl\{
   \mld(0;A_{\CC},\mathfrak b^e)
   \;\Bigm|\;
   \begin{array}[t]{@{}l@{}}
     (A_{\CC},0)\text{ is a smooth point on a surface over }\CC,\\
     \mathfrak b^e\text{ is a multi-ideal}
   \end{array}
 \Bigr\}, \\[1em]
(2)\quad
&\Bigl\{
   \lct(0;A,\mathfrak a^e)
   \;\Bigm|\;
   \begin{array}[t]{@{}l@{}}
     (A,0)\text{ is a smooth point on a surface over }k,\\
     \mathfrak a^e\text{ is a multi-ideal}
   \end{array}
 \Bigr\} \\
&\qquad\subset
 \Bigl\{
   \lct(0;A_{\CC},\mathfrak b^e)
   \;\Bigm|\;
   \begin{array}[t]{@{}l@{}}
     (A_{\CC},0)\text{ is a smooth point on a surface over }\CC,\\
     \mathfrak b^e\text{ is a multi-ideal}
   \end{array}
 \Bigr\}.
\end{aligned}
\]
\end{cor}
   
The stronger statements that the equalities in (1) and (2) hold were already proved in \cite{i2} using toric methods. Thus our theorem reproves these containments by a different method.

Campillo (\cite{cam}) proved that for every plane curve singularity $(C,0)$ in positive characteristic, there exists a complex model $(C_{\CC}, 0)$ whose equisingular type coincides with that of $(C,0)$ (see \cite{gr} for a survey).
The following is a corollary of Theorem \ref{newmain} showing that a complex model of a plane curve singularity in positive characteristic can be obtained via a skeleton lifting.

\begin{cor}\label{campillo}
       Let $C\subset \AA_k^2$ be a curve defined over an algebraically closed field $k$ of characteristic $p>0$.
      Then, there exists a lifting $\wC\subset \AA_\CC^2$ such that $(\wC, 0)$ gives  Campillo's  complex model of $C\subset \AA_k^2$.
       In other words, there are embedded resolutions 
       $$\varphi \colon A' \to \mathbb{A}^2_k\qquad \mbox{and}\qquad
       \widetilde{\varphi} \colon \widetilde{A}' \to \mathbb{A}^2_{\mathbb{C}}$$
        of $(C,\mathbb{A}^2_k)$ and
       $(\wC,\mathbb{A}^2_{\mathbb{C}})$, respectively, with the same
       multiplicity data of the reduced total transforms  at every stage of
      the successive point blow-ups (see, \cite[section 2.1]{gr}).

\end{cor}
  
    
 This paper is organized as follows. In Sections 2 and 3, we formulate the statements in arbitrary dimension for future use. In Section 2, we introduce the basic notions of liftings from positive characteristic to characteristic zero. In Section 3, we introduce log discrepancies, minimal log discrepancies, and log canonical thresholds, which are birational invariants measuring the nature of singularities of pairs. In Section 4, for a given sequence of blow-ups at closed points of an affine space of arbitrary dimension over a field of characteristic $p>0$,
 we construct a sequence of blow-ups of the complex affine space obtained by skeleton liftings and show that it satisfies property (i)--(iii) of Theorem~\ref{newmain}. We then prove (iv) of Theorem 1.1 in dimension two. In Section 5, we prove the corollaries of the main theorem.    
  \vskip.5truecm
   
  \noindent
  {\bf Acknowledgement.} The author would like to thank Lawrence Ein, Kazuhiko Kurano, Mircea Musta\c{t}\v{a}, Yuri Prokhorov,
  Kohsuke Shibata and Shoji Yokura for 
  useful discussions and comments.
  The author expresses her heartfelt thanks to J\'anos Koll\'ar for providing critical examples and comments for the preliminary versions of this paper.

\vskip.5truecm

{\bf Conventions} 
\begin{enumerate}
           \item Let $L_1$ be a ring, $L_2$  an $L_1$-algebra and $A$ a scheme of finite type over $L_1$.
       We denote the scheme $A\times_{\spec L_1}\spec L_2$ by $A\otimes_{L_1}L_2$ to avoid the 
      bulky expression.

    \item  For an integral domain $R$, we denote the field of fractions of $R$ by $Q(R)$.  
     \item A prime ideal $\p\subset R$ of a Noetherian integral domain is called a regular
     prime ideal if the local ring associated to $\p$ is a regular local ring.
     Here, note that if $R$ is a graded ring and $\p$ is a homogeneous prime ideal the 
     associated local ring is $R_{(\p)}$  by the symbol in \cite[Definition, p.116]{ha}.
 \item Let $\p\subset R$ be a regular prime ideal. 
           A set of elements $\{f_1,\ldots,f_c\}\subset \p$   is called a regular system of parameters (RSP for short)
           of $\p$ if it is an RSP in the local ring associated to $\p$.
           If $R$ is graded and $\p$ is homogeneous, the set gives an RSP in $R_{(\p)}$ by an appropriate dehomogenization.
           
            
       
\end{enumerate}

\vskip1truecm
\section{Basics on liftings}
\noindent
In this paper we construct a ``lifting" which is an  object on a variety over $\CC$ corresponding to that
 in positive characteristic.
Basic notions about liftings are discussed in   \cite{inv}, and here we  summarize some of them,
which we will use in this paper.

\begin{defn} Let $R$ be an integral domain which is finitely generated $k$-algebra,
  where $k$ is a field of characteristic $p\geq 0$.
  A subring $S\subset R$ is called a {\sl skeleton} of $R$, 
  if $S$ is finitely generated $\zp$-subalgebra of $R$ such that $k[S]=R$.
\end{defn}

\begin{defn}\label{def-of-lift-elem}

{\bf (compatible skeletons)}
   Let $k$ be a field of characteristic $p>0$ and $K$ a field of characteristic $0$.
   Let $R$ and $\wr$ be  finitely generated $k$-algebra and 
   $K$-algebra, respectively.
   Let $S\subset R$ and $\ws\subset \wr$ be skeletons.
   If the skeletons satisfy  
   \begin{center}
    $\ws$ is $\ZZ$-flat at $p$, and
   \end{center}
   $$S\simeq\ws\otimes_\ZZ \zp,$$  we call them {\sl compatible skeletons}.
   We  sometimes say that a skeleton $\ws$ is a {\sl lifting of the skeleton} $S$.
   We also call the canonical surjection $\Phi_p:\ws\to S$ {\sl compatible skeletons}.
   In this case we write $$\wr (\mod p)=R.$$ 
   If schemes $\wa$ and $A$  are defined by $\wr$ and $R$ ({\it  eg}. as $\spec$ or $\proj$)
   which have compatible skeletons, we also write 
   $$\wa (\mod p) =A.$$
  
  Here, we note that if $\wr$ is an integral domain (therefore, so is $\ws$), then $\ZZ$-flatness is automatic. 
  In this case $S\simeq\ws\otimes_\ZZ \zp$ is the only condition for $\ws$ and $S$ to be compatible skeletons.
 
   Let $\Phi_p:\ws\to S$ be compatible skeletons of $\wr$ and $R$.
   If $\tf\in \wr$ and $f\in R$ satisfy 
   $\tf\in\ws$ and $f=\Phi_p(\tf)$, then we write $\tf (\mod p)=f$ and call $\tf$ a {\sl  lifting} of $f$.
   If $\{\tf_1,\tf_2,\ldots, \tf_n \}\subset \ws$ and $f_i=\Phi_p(\tf_i)$ for all $i=1,\ldots, n$, then we 
   write $\{\tf_1,\tf_2,\ldots, \tf_n \} (\mod p)=\{f_1,\ldots, f_n\}$ and call $\{\tf_1,\tf_2,\ldots, \tf_n \} $
   a lifting of $\{f_1,\ldots, f_n\}$.
 
 \end{defn}

When we think of liftings of elements, we always use a common pair of  compatible skeletons in a continuous 
discussion.

\begin{ex}\label{example-liftings}
    The following are examples of  liftings.
    The first one is a trivial example and the second and third are more practical and will play basic roles
    in this paper.
    Assume that $k$ is a field of characteristic $p>0$.
    
\noindent
{\bf(1)} 
     We have $\CC(\mod p)= k$ with compatible skeletons $S=\zp$ and $\ws=\ZZ$.

\noindent    
{\bf(2)} 
      For any  finite elements $a_1,\ldots,a_r\in k$
      there exist  liftings $\ta_1,\ldots,\ta_r\in \CC$.

      Indeed compatible skeletons $\Sigma\subset k$ and $\widetilde\Sigma\subset \CC$ 
      giving these liftings are constructed  in \cite[Proposition 2.3]{inv}.
However, since this statement plays a fundamental role in the paper, we restate the proof here.
\begin{proof}
    Considering the subring $\zp[a_1,\ldots, a_n]\subset k$,
    we obtain canonical surjections:
    $$R:=\ZZ[Y_1,\ldots,Y_n]\stackrel{\psi}\longrightarrow S:=\zp[Y_1,\ldots,Y_n]
    \stackrel{\varphi}\longrightarrow \zp[a_1,\ldots,a_n]$$
    with $Y_i \mapsto a_i$.
    Let $P:=Ker \varphi \subset S$ and $Q:=Ker \varphi\circ\psi \subset R$,
    then these are prime ideals in  regular rings.
    Therefore $R_Q$ and $S_P$ are also regular local rings.
    Hence we obtain
    \begin{enumerate}
       \item $f_1,\ldots,f_c\ (\in P)$ which form a regular system of parameters of $S_P$,
       \item  $\tf_1,\ldots, \tf_c, p \  (\in Q)$ which form a regular system of parameters of $R_Q$ and
       \item  $\psi(\tf_i)=f_i $ for $i=1,\ldots, c$.
    \end{enumerate}
    Then, $R_Q/(\tf_1,\ldots,\tf_c)R_Q$ is also a regular local ring, in particular it is an
    integral domain.
    By considering the homomorphism $\ZZ\to R_Q $ and the regular sequence 
    $\{f_1,\ldots,f_c \}\subset S_P= R_Q\otimes_\ZZ\zp$,
    we obtain that $R_Q/(\tf_1,\ldots,\tf_c)R_Q$ is flat over $\ZZ$
    by \cite[Corollary, p.177]{mats}.
    In particular, the homomorphism $\ZZ \to R_Q/(\tf_1,\ldots,\tf_c)R_Q$ is injective,
    which implies the ring $R_Q/(\tf_1,\ldots,\tf_c)R_Q$ is of characteristic $0$.
 
    Note that the algebras $\wsig':=R_Q/(\tf_1,\ldots,\tf_c)R_Q$ and $\Sigma':=S_P/(f_1,\ldots,f_c)S_P$
    satisfy $$\wsig'\otimes_\ZZ\zp=\Sigma'=\zp(a_1,\ldots, a_n).$$ 
  However,  these $\wsig'$ and $\Sigma'$ are not appropriate for the skeletons,
  because these are not finitely generated as $\ZZ$-algebra and $\zp$-algebra, respectively.
  So we replace $\wsig'$ and $\Sigma'$ by smaller ones.

    As $PS_P=(f_1,\ldots,f_c)S_P$, there exists $h\in S\setminus P$ such that 
    $PS_h=(f_1,\ldots,f_c)S_h$.
    Take $\th\in R\setminus Q$ such that $\psi(\th)=h$ and let
    $$\widetilde\Sigma:=R_\th/(\tf_1,\ldots,\tf_c)R_\th.$$
   This, as well, is an integral domain and of characteristic $0$.
    Now noting that $S/P=\zp[a_1,\ldots, a_c]$, we obtain
    $$\Sigma:=\widetilde\Sigma\otimes_\ZZ\zp=\frac{R_\th/(\tf_1,\ldots,\tf_c)R_\th}
    {p({R_\th/(\tf_1,\ldots,\tf_c)R_\th})}
    =S_h/(f_1,\ldots,f_c)S_h$$
    $$=(S/P)_h\subset 
    \zp(a_1,\ldots,a_c) \subset k.$$
  By the definitions, $\wsig$ and $\Sigma$ are finitely generated as $\ZZ$-algebra and as $\zp$-algebra, respectively.
    By the surjection $\Phi_p:\widetilde\Sigma\to \widetilde\Sigma\otimes_\ZZ\zp=\Sigma$
    we can take $\widetilde{a}_1, \ldots, \widetilde{a}_n\in \widetilde\Sigma$
    corresponding to $a_1,\ldots, a_n\in \Sigma$.
     Now, by the definition of $\widetilde\Sigma$, the field $K_0:=Q(\widetilde\Sigma)$ of fractions of $\widetilde\Sigma$ is
     a finitely generated field extension of $\QQ$.
     Then, by Baby Lefschetz Principle (see, for example, \cite[Proposition 4]{tao}), 
     there is an isomorphism into the subring:
     $$\phi: K_0 \stackrel{\sim}\longrightarrow \phi(K_0)\subset \CC.$$
     Then, we obtain 
     $\{\widetilde a_1,\ldots, \widetilde a_n\}\subset \CC$
      and $\{\widetilde a_1,\ldots, \widetilde a_n\} (\mod p)=\{a_1,\ldots, a_n\}$.
      As is seen, $\wsig\subset \CC$ and $\Sigma\otimes \zp$ are compatible skeletons.
\end{proof}
     
\noindent      
{\bf(3)}  Let $\Sigma\subset k$ and $\wsig \subset \CC$ be  compatible skeletons.
      Then, $S:=\Sigma[x_1,\ldots,x_N]\subset 
      k[x_1,\ldots,x_N]$ and  $\ws:= \wsig[x_1,\ldots,x_N]\subset \CC[x_1,\ldots,x_N]$
      are  also compatible skeletons for $\CC[x_1,\ldots,x_N] (\mod p)= k[x_1,\ldots,x_N]$
      in the canonical way.
      Explicitly, a polynomial $f\in S$ has a lifting $\tf\in \ws$ in the following form:
          $$f=\sum_{i\in \ZZ^N_{\geq 0}} a_i{\bold{x}}^i\ \ \  ,\ \ \mbox{then}\ \ \ \tf=\sum_{i\in \ZZ^N_{\geq 0}} \ta_i\bold{x}^i,$$
    where $\ta_i$ is a lifting of $a_i$.
\end{ex}

\begin{defn} Let $R$ and $\wr$ be finitely generated algebras 
over  fields $k$ and $K$ of characteristic $p>0$ and of 0, respectively.

Assume $\wr (\mod p)=R$ holds by compatible skeletons $\Phi_p: \ws\to S$, where $S\subset R$ and 
        $\ws\subset \wr$.
        
        For  ideals $\a_R\subset R$ and $\taa_\wr\subset \wr$  we write $\taa_\wr(\mod p)=\a_R$ and say
       {\sl  `` $\taa_\wr $ is a lifting of $\a_R$" } if we can take a generator system of $\taa_\wr$ from $\ws$ whose image by $\Phi_p$
        generates
         $\a_R\subset R$. It is equivalent to the following:

          \begin{enumerate}
\item There are ideals $\a\subset S$ and $\taa\subset \ws$
       such that $\taa_\wr=\taa\wr$, $\a_R=\a R$, and  
\item         $\Phi_p(\taa)=\a$.
\end{enumerate}
 In this case, we call $\a$  a {\sl skeleton of }$\a_R$, and $\taa$  a {\sl skeleton of} $\taa_\wr$.
                
\end{defn}

\begin{rem} According to the definition above, the unit ideal $\wr$ itself is a lifting of any ideal $\a_R\subset R$ generated by
 elements of the skeleton $S\subset R$.
 Because the pull-back $\taa:=\Phi_p^{-1}(\a)$ satisfies the conditions above and contains $p\in \ZZ\subset \wr$,
 a unit element of $\wr$.

\end{rem}
\noindent
This gives a trivial example that
 a lifting of a prime ideal in $R$ is not necessarily prime. 
 For a prime ideal, we define a ``prime lifting''
 in the skeleton rings.

 \begin{defn} \label{def-of-prime-lift}
      Let $S\subset R$ and 
        $\ws\subset \wr$ be compatible skeletons  for finitely generated algebras $R$ and $\wr$  
over  fields $k$ and $K$ of characteristic $p>0$ and of 0, respectively.
Let $\p\subset S$ and $\wp\subset \ws$ be prime ideals.
 We call $\wp$ a {\sl  prime lifting of $\p$ in the skeleton level}  if $\Phi_p(\wp)=\p$ and $\height \p=\height \wp$.

  Sometimes we omit the term ``in the skeleton level'' when it is clear that we are working in the skeletons.
\end{defn}   
    

In general, a prime lifting in the skeleton level does not exist.
However, it does exist under certain conditions.


\begin{prop}\label{lift-of-maximal-ideal}
Let $k$ be an algebraically closed field of characteristic $p>0$.
      Let $R$ be a finitely generated $k$-algebra of dimension $N\geq2$, and 
      $\wr$ a finitely generated $\CC$-algebra.
        Assume 
        $$\wr (\mod p)=R$$
         with the compatible skeletons $S\subset R$ and $\ws\subset \wr$.
         
   Let $X:=\spec R$,  and $x\in X$ a regular closed point.
   Assume 
   the defining ideal $\m_x\subset R$ is generated by RSP $f_1,\ldots,f_N $ of $\m=\m_x\cap S$,
   then $\m_x$
   has a prime lifting in the skeleton level.
\end{prop}
\begin{proof} For each $1\leq i\leq N$, take a lifting, $\tf_i\in \ws$ of $f_i$.
     As $\{f_1,\ldots,f_N \}$ is a regular sequence in $S_\m$, the lifting 
     $$\{\tf_1,\ldots, \tf_N\}\subset \ws_{\Phi_p^{-1}(\m)}$$
     forms a regular sequence and also
   \[
W:=\widetilde{S}_{\Phi_p^{-1}(\mathfrak m)}/(\widetilde f_1,\dots,\widetilde f_N)
\]
is $\mathbb Z$-flat at $p$ by [16, Corollary p.~177]. Since
\[
W/pW \simeq W\otimes_{\mathbb Z}\mathbb Z/(p)
 \simeq S_{\mathfrak m}/(f_1,\dots,f_N)
\]
is an integral domain and $W$ is $\mathbb Z$-flat at $p$, multiplication by $p$
is injective on $W$. Moreover, $W$ is Noetherian local and $p$ belongs to its
maximal ideal, so
\[
\bigcap_{n\ge 0} p^nW = 0
\]
by Krull's intersection theorem. We claim that $W$ is an integral domain.
Indeed, if $ab=0$ with $a,b\neq 0$, then by the above intersection property we
can write
\[
a=p^r a', \qquad b=p^s b'
\]
with $a',b'\notin pW$. Since $p$ is a non-zero-divisor on $W$, it follows that
$a'b'=0$. Reducing modulo $p$, we obtain
\[
\overline{a'}\,\overline{b'}=0 \qquad \text{in } W/pW,
\]
which is impossible because $W/pW$ is an integral domain. Thus $W$ is an
integral domain. Therefore $(\widetilde f_1,\dots,\widetilde f_N)$ is a prime
ideal in $\widetilde{S}_{\Phi_p^{-1}(\mathfrak m)}$.
     Define $\wm:=(\tf_1,\ldots, \tf_N)\ws_{\Phi_p^{-1}(\m)}\cap \ws$,
     then it is a prime ideal of $\ws$ of height $N$ and $\Phi_p(\wm)=\m$.
\end{proof}

\begin{cor}\label{inclusions}
    Under the condition of Proposition \ref{lift-of-maximal-ideal}, assume $H\subset X$ is a smooth 
    subvariety, and its defining ideal  $I_H$ is generated by a part of RSP of $\m:=\m_x\cap S$ 
    and $I_H$ has a prime lifting $\wI\subset \ws$ in the skeleton level.
    If a closed point $x\in X$ is contained in $H$, {\it i.e.,}
 $$I_H\subset \m_x,$$   
    then there is a prime lifting $\wm\subset \ws$ of the 
   defining  ideal $\m_x$ of $x$, such that $$\wI\subset \wm.$$ 
\end{cor}

\begin{proof}
      By the assumption that $I_H$ has a prime lifting,
      we may assume that $I_H\subset R$ is generated by $I=I_H\cap S\subset S$
      and $\m_x$ is generated by $\m=\m_x\cap S$.
      Then, $S'=S/I$ is a skeleton of $R/I_H$ which is the affine ring of smooth variety $H$.
      We can regard $\ws'=\ws/\wI$ as a compatible skeleton of $S'$.
      Apply Proposition \ref{lift-of-maximal-ideal} to $\m'=\m/I\subset S'$,
      then we obtain its prime lifting 
      $$\wm' \subset \ws',$$
      which yields a prime lifting $\wm=\pi^{-1}(\wm')$ of $\m$ satisfying $\wm\supset \wI$,
      where $\pi: \ws\to \ws'$ is the canonical surjection.

\end{proof}


In this paper, we work on algebras of the special type: {\it i.e.,} a finitely generated $k$-subalgebra $R$ of a polynomial ring.
On such an algebra $R$, properties on an ideal $\a$ of $R$ are reduced to the properties of its skeleton ideal $\a\cap S$ as follows:

\begin{lem}\label{ext-of-field}
Let $K$ be a field, let $\Sigma$ be a skeleton of $K$, and let
$Q$ be the field of fractions of $\Sigma$.
Let
\[
S \subset \Sigma[x_1,\ldots,x_m]
\]
be a finitely generated $\Sigma$-subalgebra, and let $\mathfrak a \subset S$ be an
ideal such that $\mathfrak a \cap \Sigma = \{0\}$.
Let
\[
Q[S]\subset Q[x_1,\ldots,x_m]
\qquad\text{and}\qquad
K[S]\subset K[x_1,\ldots,x_m]
\]
be the $Q$- and $K$-subalgebras generated by $S$, respectively, and let
\[
\mathfrak a_Q \subset Q[S],
\qquad
\mathfrak a_K \subset K[S]
\]
be the ideals generated by $\mathfrak a$.
Then the following hold.

\begin{enumerate}
\item[(i)]
We have
\[
\operatorname{ht}\mathfrak a
\le
\operatorname{ht}_+\mathfrak a
=
\operatorname{ht}\mathfrak a_Q
=
\operatorname{ht}\mathfrak a_K,
\]
where
\[
\operatorname{ht}_+\mathfrak a
:=
\min\bigl\{
\operatorname{ht}\mathfrak p
\;\big|\;
\mathfrak p\in \operatorname{Ass}_S(S/\mathfrak a),
\ \mathfrak p\cap \Sigma=\{0\}
\bigr\}.
\]

\item[(ii)]
If $\mathfrak a$ is a prime ideal, then so is $\mathfrak a_Q$.
In this case,
\[
\operatorname{ht}\mathfrak a
=
\operatorname{ht}\mathfrak a_Q
=
\operatorname{ht}\mathfrak a_K.
\]
Moreover, if $\mathfrak a_K$ is a regular prime ideal of $K[S]$, then
$\mathfrak a$ is a regular prime ideal of $S$.

\item[(iii)]
Assume that $\mathfrak a$ is prime.
Let $f\in S\setminus \Sigma$ and $n\in \mathbb Z_{>0}$.
Then the following are equivalent:
\begin{enumerate}
\item[(a)]
\[
f\in \mathfrak a^n S_{\mathfrak a}\setminus \mathfrak a^{n+1}S_{\mathfrak a};
\]
\item[(b)]
\[
f\in \mathfrak a_Q^n (Q[S])_{\mathfrak a_Q}
\setminus
\mathfrak a_Q^{n+1}(Q[S])_{\mathfrak a_Q};
\]
\item[(c)]
if $\mathfrak a_K$ is prime, then
\[
f\in \mathfrak a_K^n (K[S])_{\mathfrak a_K}
\setminus
\mathfrak a_K^{n+1}(K[S])_{\mathfrak a_K}.
\]
\end{enumerate}

\item[(iv)]
Assume that $\operatorname{char}K=0$ and that $\mathfrak a_Q$ is a locally
principal prime ideal.
Then, for every $f\in S\setminus \Sigma$ and every $n\in \mathbb Z_{>0}$,
conditions {\rm (a)} and {\rm (b)} above are also equivalent to:
\begin{enumerate}
\item[(d)]
for every minimal prime $\mathfrak p$ of $\mathfrak a_K$, the ideal $\mathfrak p$
is locally principal at its generic point and
\[
f\in \mathfrak p^n K[S]_{\mathfrak p}
\setminus
\mathfrak p^{n+1} K[S]_{\mathfrak p}.
\]
\end{enumerate}
\end{enumerate}
\end{lem}

\begin{proof}
Since $Q[S]=S_{\Sigma\setminus\{0\}}$, localization gives
\[
\operatorname{Ass}_{Q[S]}(Q[S]/\mathfrak a_Q)
=
\bigl\{
\mathfrak p_Q
\;\big|\;
\mathfrak p\in \operatorname{Ass}_S(S/\mathfrak a),\ 
\mathfrak p\cap\Sigma=\{0\}
\bigr\}.
\]
Hence
\[
\operatorname{ht}_+\mathfrak a=\operatorname{ht}\mathfrak a_Q.
\]
On the other hand, $Q\hookrightarrow K$ is faithfully flat, so the induced map
\[
Q[S]\hookrightarrow K[S]=K\otimes_Q Q[S]
\]
is also faithfully flat. Therefore, by \cite[III, Corollary 9.6]{ha}
(or algebraically by \cite[Theorem 15.1]{mats}),
the height of a minimal prime of $\mathfrak a_Q$ agrees with the height of a
minimal prime of its extension to $K[S]$. Thus
\[
\operatorname{ht}\mathfrak a_Q
=
\operatorname{ht}(\mathfrak a_QK[S])
=
\operatorname{ht}\mathfrak a_K.
\]
This proves the equalities in {\rm (i)}, and the inequality
$\operatorname{ht}\mathfrak a\le \operatorname{ht}_+\mathfrak a$ is obvious.

For {\rm (ii)}, if $\mathfrak a$ is prime, then $\mathfrak a_Q$ is prime as a
localization of $\mathfrak a$.
The height equalities follow immediately from {\rm (i)}.
Moreover, the map
\[
S\hookrightarrow K[S]
\]
is flat, being the base change of the flat homomorphism $\Sigma\hookrightarrow K$.
Hence, if $\a_K$ is a regular prime ideal, then
$\a=\a_K\cap S$ is also regular
(see, for example, \cite[Theorem 15.1(ii)]{mats}).

We next prove {\rm (iii)}.
Since $\Sigma\setminus\{0\}\subset S\setminus \a$, localization yields
\[
(Q[S])_{\mathfrak a_Q}\cong S_{\mathfrak a},
\qquad
\mathfrak a_Q (Q[S])_{\mathfrak a_Q}=\mathfrak a S_{\mathfrak a}.
\]
Therefore, for every $n\ge 1$,
\[
\mathfrak a_Q^n (Q[S])_{\mathfrak a_Q}
=
\mathfrak a^n S_{\mathfrak a},
\]
and hence {\rm (a)} and {\rm (b)} are equivalent.

Assume now that $\mathfrak a_K$ is prime.
Since $Q\hookrightarrow K$ is faithfully flat, we have
\[
\mathfrak a_K^n\cap Q[S]=\mathfrak a_Q^n
\qquad (n\ge 1).
\]
It follows that the $\mathfrak a_Q$-adic order of $f$ is the same as the
$\mathfrak a_K$-adic order of $f$, so {\rm (b)} and {\rm (c)} are equivalent.

Finally, we prove {\rm (iv)}.
Since $Q[S]\hookrightarrow K[S]$ is flat, every minimal prime $\mathfrak p$ of
$\mathfrak a_K$ satisfies
\[
\operatorname{ht}\mathfrak p=\operatorname{ht}\mathfrak a_Q
\]
by \cite[Theorem 15.1]{mats}.
Because $\mathfrak a_Q$ is locally principal and prime, the local ring
$Q[S]_{\mathfrak a_Q}$ is a regular local ring of dimension one.
Hence $Q\to Q[S]$ is smooth around $\mathfrak a_Q$, and therefore the base change
\[
K\to K[S]
\]
is also smooth around every minimal prime $\mathfrak p$ of $\mathfrak a_K$.
In particular, each such $\mathfrak p$ is regular of height one, and hence is
principal at its generic point.

Now the natural local homomorphism
\[
Q[S]_{\mathfrak a_Q}\to K[S]_{\mathfrak p}
\]
is flat, and since $\operatorname{char}K=0$, the maximal ideal of
$K[S]_{\mathfrak p}$ is the extension of the maximal ideal of
$Q[S]_{\mathfrak a_Q}$ (see, for example, \cite[Lemma 33.5.3]{stack}).
Thus a generator of $\mathfrak a_QQ[S]_{\mathfrak a_Q}$ becomes a generator of
$\mathfrak pK[S]_{\mathfrak p}$.
Consequently,
\[
f\in \mathfrak a_Q^n (Q[S])_{\mathfrak a_Q}
\quad\Longleftrightarrow\quad
f\in \mathfrak p^n K[S]_{\mathfrak p},
\]
and similarly for $n+1$.
Therefore {\rm (b)} and {\rm (d)} are equivalent.
\end{proof}

\begin{rem} Lemma \ref{ext-of-field} 
also holds  for the case that $S \subset \Sigma[x_1,\ldots, x_N]$ has a graded algebra structure and ideals are  homogeneous ideals,
where the degree is given by a weighted degree of $x_i$'s.
\end{rem}

 %
 %
 


%


\begin{lem}\label{coef-extension} Let $k$ be a field, and $\Sigma\subset k$ a skeleton.
 Let $R$ be a finitely generated $k$-subalgebra of $ k[x_1,\ldots,x_m]$ and
   $S\subset R$ a skeleton such that it is a $\Sigma$-subalgebra of    $\Sigma[x_1,\ldots,x_m]$.
  Let $$f_i\in R\ \ (i=1,\ldots, n)$$
   be any finite number of elements of $R$.
  Then, by replacing $\Sigma$ by a larger skeleton $\Sigma'\supset \Sigma$, and $S$ by $S':=\Sigma'[S]$.
  We obtain a new skeleton $S'\subset R$ such that 
  $$f_i \in S'\ \ \  (i=1,\ldots, n).$$ 
\end{lem}

\begin{proof}
   As $S$ is finitely generated over $\ZZ/(p)$, where $p=\cha \ k$, $S$ is  finitely generated also over $\Sigma$
    so that we can express,  $S=\Sigma[s_1,..,s_n]$ with $s_1,\ldots s_n\in S$.
  Then, $R=k [s_1,..,s_n]$, because $R=k[S]$.
 Therefore, an element  $f \in R$ is expressed as 
 $$ f = F(s_1,...,s_n),\ \mbox{where}\ \   F(Y_1,\ldots,Y_n)\in k[Y_1,\ldots,Y_n],$$
   We call the coefficients of $F$ ``coefficients of $f$."
      Let $\{ a_i \}\subset k$ be the set of all coefficients of $F$.
      Note that it is a finite set.
      If these  already belong to $\Sigma$, then  $f \in S$.
      If not, let $\Sigma'$ be $\Sigma[ a_i ]$.
      Then, $S' := \Sigma' [S]$ becomes a larger skeleton of $R $ with $f\in S'$.
      In this way, adding the coefficients of the finite number of elements $\{f_i\}$ of $R$ to the original $\Sigma$,
      we obtain a new skeleton $S$ of $R$ containing all $\{f_i\}$.
      \end{proof}
This extension method is used repeatedly throughout the paper.
When we use this method, we say ``by adding the coefficients'' or ``by enlarging the coefficient skeleton''.

\begin{rem} Let $\Sigma\subset k$ be a skeleton.
If the original skeleton $S\subset \Sigma[x_1,\ldots,x_m]$ has a lifting skeleton 
$\ws\subset \wsig[x_1,\ldots,x_m]$, 
then an enlarging $\Sigma'\supset \Sigma$ of the coefficient skeleton 
induces an enlarged skeleton $\wsig'\supset \wsig$ by Example \ref{example-liftings}, which yields enlarged
compatible skeletons 
$$\ws':=\wsig'[\ws]\to S':=\Sigma'[S]$$ 
\end{rem}


\begin{prop} \label{height-of-skeletons}
       Let $k$ be a field of characteristic $p>0$.
        Let $R\subset k[x_1,\ldots, x_m]$ be a finitely generated $k$-subalgebra and
        $\wr\subset \CC[x_1,\ldots, x_m]$ a finitely generated $\CC$-algebra.
        Assume $\wr (\mod p)=R$ with the compatible skeletons $S\subset R$ and $\ws\subset \wr$ 
        such that these are $\Sigma$-subalgebra and $\wsig$-subalgebra of $\Sigma[x_1,\ldots, x_m]$ and
        $\wsig[x_1,\ldots, x_m]$, respectively.
        Let $\a\subset R$ be a non-zero proper ideal and $\taa\subset \wr$ its lifting with respect to 
        these compatible skeletons $S$ and $\ws$.
                Then, $$\height \a\leq \height \taa.$$
\end{prop}

\begin{proof}  Let $\p_{1},\ldots, \p_{s}$ be the minimal associated primes of $\a\subset R$.
    Then, $\underline{\p_i}:=\p_{i}\cap S$'s are the minimal associated primes of $\underline\a=\a\cap S$
    and satisfy 
     $\underline{\p_i }\cap \Sigma=\{0\}$, since otherwise $\p_{i}$ would contain  a unit in $R$, a contradiction.
     Here, we may assume that $\Sigma$ contains all coefficients of the generators of $\underline\p_i$'s by  enlarging 
     the coefficient skeleton $\Sigma$
     as in Lemma \ref{coef-extension}.
    Therefore, $\p_i\subset R$ is the extensions of $\underline\p_i\subset S$, and by Lemma \ref{ext-of-field}, (i), it follows
    $$\height \a=\height\underline\a.$$
    On the other hand, let $\underline\taa\subset \ws$ be the skeleton ideal of $\taa\subset \wr$, and
    $\wp$  a minimal associated prime of $\taa\subset\wr$ such that $\height \taa=\height \wp$.

    Then, as $\wr$ is flat over $\ws$, it follows
 \begin{equation}\label{height1}   
    \height \wp\geq\height \underline\wp,
 \end{equation}   
    where $\underline\wp=\wp\cap \ws$ (see, for example, \cite[Theorem 15.1, (ii)]{mats}).
    As $\underline\wp\cap \wsig=\{0\}$, a fortiori $\underline\wp\cap \ZZ=\{0\}$ holds,
    which yields 
\begin{equation}\label{height2}       
    \height\underline\wp=\height(\underline\wp\otimes_\ZZ\QQ)=\height(\underline\wp\otimes_\ZZ\zp)
    =\height \Phi_p(\underline\wp),
\end{equation}    
    where the second equality follows from the flatness of $\ws$ over $\ZZ$.
    Here, since $\Phi_p(\underline\wp)\supset \Phi_p(\underline\taa)=\underline\a$, it follows  
\begin{equation}\label{height3}       
    \height \Phi_p(\underline\wp)\geq
    \height \underline\a.
\end{equation}
By connecting inequalities (\ref{height1}), (\ref{height2}) and (\ref{height3}),
we obtain $\height \taa\geq \height \a$.    
   \end{proof}
\vskip.5truecm

\section{log discrepancies for multi-ideals and log canonical thresholds of  ideals}
\noindent
   In this section we study a pair $(A, \a^e)$ consisting of a smooth 
   affine variety $A$ of dimension $N$  defined
   over a field of arbitrary characteristic
      and a ``multi-ideal" $\a^e=\a_1^{e_1}\cdots\a_r^{e_r}$ on $A$, 
      where $\a_i$'s are non-zero coherent ideal sheaves on $A$ and $e$ is a combination of  the exponents 
     $e=(e_1,\ldots,e_r)\in \RR_{>0}^r$.
     
\begin{defn}\label{p.d.} Let $A$ be a normal 
variety defined over  a field $k$, and let $P\in A$ be a (not necessarily closed) point.

Let  \( \varphi_1:A_1 \to A \) be
a proper birational morphism with  $A_1$ normal and $E_1\subset A_1$ 
an irreducible divisor.
 Let  \( \varphi_2:A_2 \to A \) be another
 proper birational morphism from a normal $A_2$ with an irreducible divisor $E_2\subset A_2$.
 Define 
 $$E_1\sim E_2$$
  if  the birational  map
  \(\varphi_2^{-1}\circ \varphi_1 : A_1  \dasharrow  A_2 \) is  a local isomorphism at the 
  generic points of \( E_1 \) and $E_2$.
 Then, $\sim$ becomes an equivalence relation, and
  such an equivalence class  is called a {\sl prime divisor over $A$}. 
    In this case we denote $E_1$ and
   \( E_2 \) by the same symbol, say $E$, and we say ``$E$ appears on $A_1$ and on $A_2$''.
   (Strictly speaking, we should be talking about the
corresponding valuation instead.)
 If $P$ is the generic point of  $\varphi_1(E)$, then we call $E$  a {\sl prime divisor over $A$ with the center at $P$}
 or just a {\sl prime divisor over $(A,P)$}.
\end{defn}

\begin{defn}\label{defoflogcano}
   Let $(A, \a^e)$ be a pair as in the beginning of this section.
    The {\sl log discrepancy} of such a pair 
    at a prime divisor $E$ over $A$ is 
    defined as 
$$a(E; A,\a^e):=k_E+1-\sum_{i=1}^r e_i\cdot  v_E(\a_i),$$
      where $k_E$ is the coefficient of the relative canonical divisor 
      $K_{{A}'/A}$ at $E$.\\
Here, $\varphi:{A'}\to A$ is a birational morphism such that $E$ appears on 
  a normal variety 
   ${A'}$. 
   The valuation of an ideal $\a_i$ at $E$ is defined as follows:
   $$ v_E(\a_i):=\min\{ v_E(f) \mid f\in \a_i\}$$

   We say that the pair $(A, \a^e)$ is log canonical  at a (not necessarily closed)
     point  $P\in A$ if
\begin{equation}\label{discrepancy}
a(E;A,\a^e)\geq 0,
\end{equation}
holds for every exceptional prime divisor $E$ over $A$ whose center contains $P\in A$.
\end{defn}

\begin{defn}\label{defoflct}
   Let $A$ be a smooth variety, $P\in A$ a  point and $\a^e$ a non-zero multi-ideal on $A$.
   We define the {\sl log canonical threshold} of the pair $(A,\a^e)$ at $P$ as follows:
   $$\lct (P; A, \a^e)= \sup\{ c \in \RR_{>0} \mid (A, (\a^e)^c)\ \mbox{is\ log \ canonical\ at \ }P\}.$$
   
   Let $E$ be a prime divisor over $A$ whose center on $A$ is contained in the locus of $\a$.
   Define 
   $$z(E; A, \a^e):=\frac{k_E +1}{v_{E}(\a^e)},$$
   where, we define $v_E(\a^e):=\sum_i e_i\cdot v_E(\a_i)$.
   Then, the following is well known:
   $$\lct(P; A,\a^e)=
   \inf \{z(E; A, \a^e) \mid E\ \mbox{is\ a\ prime\ divisor\ over\ }A \ \mbox{with\ the \ center\ containing\ }P\}.$$
\end{defn}

\begin{defn}\label{defofmld}
    Let $(A,\a^e)$ and $P\in A$ be  as in Definition \ref{defoflogcano}.
    Then, the minimal log discrepancy is defined as follows:
    
 \begin{enumerate}
\item When $\dim A\geq 2$,
    $$\mld (P; A, \a^e)
    :=\inf\{ a(E;A, \a^e) \mid  E :\ {\operatorname { prime\ divisor \ over }}\ (A, P) \}.$$ 
\item When $\dim A=1$, define $\mld (P; A,\a^e)$ 
     by the same definition as in (1) if the right hand side
     of the equality (1) 
     is non-negative and otherwise define $\mld(P; A, \a^e)= -\infty$.
\end{enumerate}
    Here, we remark that for every pair $(A, \a^e)$, either $\mld(P; A, \a^e)\geq 0$ or
    $\mld(P; A, \a^e)=-\infty$ holds.

\end{defn}

 \begin{defn} 
    Let $A$, $N$,  and $\a^e$ and  $P\in A$ be as in Definition \ref{defoflogcano}.
    We say that a prime divisor $E$ over $(A,P)$  
    computes $\mld(P; A, \a^e)$,
    if 
    
    $$\left\{\begin{array}{ll}
       a(E; A, \a^e)=\mld(P; A, \a^e),& \mbox{when}\ \ \mld(P; A, \a^e)\geq 0\\
       a(E; A, \a^e)<0,  &    \mbox{when}\ \ \mld(P; A, \a^e)=-\infty\\ 
    \end{array}\right. $$
    Let $A$, $P$ and $\a^e$ be as in Definition \ref{defoflct}.
    Let $E$ be a prime divisor  over $A$ with the center containing $P$.
    We say that a prime divisor $E$ computes $\lct(P;A,\a^e)$ if
    $$z(E;A,\a^e)=\lct (P;A,\a^e).$$
\end{defn}

   


\begin{defn}
  
 Let \( X \) be a scheme of finite type over a field $k$ and  $k'\supset k$ a field extension.
For  \( m\in \ZZ_{\geq 0} \) a \( k \)-morphism \( \spec k'[t]/(t^{m+1})\to X \) is called an  {\it{\( m \)-jet}} of \( X \) and 
 \( k \)-morphism \( \spec k'[[t]]\to X \) is called an {\sl {arc}} of \( X \).
\end{defn}


Let 
 \( X_{m} \) be the {\sl space of \( m \)-jets} 
   of  \( X \). 
   It is well known that $X_m$ has a scheme structure of finite type over $k$.
   If $X$ is an affine variety, then there exists the projective limit $$X_\infty:=\lim_{\overleftarrow {m}} X_m$$
   and it is called the {\sl space of arcs} or the {\sl space of $\infty$-jet} of $X$.
\begin{defn}
    Denote the canonical truncation morphisms induced from $k[[t]]\to k[t]/(t^{m+1})$ 
    and $k[t]/(t^{m+1})\to k$  by
    $\psi_m: X_\infty\to X_m$ and $\pi_m: X_m\to X$, respectively.
    In particular we denote the  morphism  $\psi_0=\pi_\infty : X_\infty \to X$ by $\pi$.
    
\end{defn}

\begin{defn} Let  $\a$ be a non-zero ideal on 
       a variety $X$.
        We define the subsets ``{\sl contact loci of }$\a$''  in the space of arcs  as follows:
       $$\cont^{\geq m}(\a)=\{\gamma \in X_\infty \mid \ord_\gamma(\a):=\ord_t  
       \gamma^*(\a)\geq m\},$$
       where $\gamma^*:\oo_X\to k'[[t]]$
      is the  homomorphism of rings corresponding to $\gamma$.
\end{defn}


\begin{ex}\label{equation}
     Let $Z$ be a closed subscheme of affine $N$-space
    $A:=\AA_k^N\\ =\spec k[x_1,\ldots,x_N]$ defined over a field $k$.
    Let 
    $\a\subset k[x_1,\ldots,x_N]$ be the defining ideal of $Z$ in $A$.
    Assume $\a$ is generated by $f_1,\ldots, f_r\in k[x_1,\ldots,x_N]$.
    We define polynomials $$F_i^{(j)}\in k\left[x_\ell^{(q)}\middle| 1\leq \ell\leq N, 0\leq q \leq j\right]$$
    so that 
    $$f_i\left(\sum_{q\geq 0} x_1^{(q)}t^q, \sum_{q\geq 0} x_2^{(q)}t^q, \ldots, \sum_{q\geq 0} x_N^{(q)}t^q\right)=
    F_i^{(0)}+F_i^{(1)}t +\cdots + F_i^{(j)}t^j+ \cdots.$$
    
    Then, the contact locus $\cont ^{\geq m}\a\subset A_\infty$ is defined  
    by the ideal of $k[x_\ell^{(q)}\mid 1\leq \ell\leq N, 0\leq q ]$
    generated by
    $$F_i^{(j)}\ \  (i=1,\ldots, r, j=0, 1,\ldots, m-1).$$
    Here, we note that if all coefficients of $f_i$'s are in a subring $\Sigma\subset k$,
    then all coefficients of   $F_i^{(j)}$'s are also in $\Sigma$.  
    (This fact will be used in the proof of Lemma \ref{inequality of mld}.)
    
    The fiber $\pi_m^{-1}(0)$ of the origin $0\in \AA_k^N=A$ 
    by the truncation morphism $\pi_m: A_m\to A$ is
    defined by the ideal
    $$(x_1^{(0)},\ldots,x_N^{(0)})\subset 
    k[x_\ell^{(q)}\mid 1\leq \ell\leq N, 0\leq q \leq m].$$
\end{ex}

\vskip.5truecm
The following Proposition \ref{formula} and Lemma \ref{inequality of mld} 
require the base field $k$ to be perfect. 
Since $k$ is algebraically closed, we can apply these formulas.

\begin{prop}
\label{formula}  Let $A$ be a smooth variety  defined over a perfect field,
   $P\in A$ a  closed  point and $e=(e_1,\ldots, e_r)
  \in \RR_{>0}^r$. 
  Let $\a^e$ be a multi-ideal $\a_1^{e_1}\cdots \a_r^{e_r}$ on $A$. 
  Then, it follows:

\begin{equation}\label{mld}
\mld(P; A, \a^e) 
=\inf_{m\in \ZZ_{\geq0}^r}\left\{\codim\left(\bigcap_i (\cont^{\geq m_i}(\a_i)\cap \pi^{-1}(P), A_\infty\right)-\sum_{i=1}^r m_ie_i
 \right\}.
 \end{equation}

 \begin{equation}\label{lct}
 \lct(P; A, \a^e)=\inf_{m\in \ZZ_{\geq0}}\left\{ \frac{\codim_P\left(\bigcap_i\cont^{\geq m_i}(\a_i), A_\infty\right)}{\sum_i m_ie_i}\right\},
 \end{equation}
 where $\codim_P(\cont^{\geq m}(\a), A_\infty)$ is  the minimum of  codimensions of  the irreducible components $T$ of 
 $\cont^{\geq m}(\a)$ in $ A_\infty$
 such that $\pi(T) \cap \{P\}\neq \emptyset$.

\end{prop}   
\begin{proof}
The equality (\ref{mld}) is obtained by
\cite{EM} for characteristic $0$ and by  \cite{ir2} for arbitrary characteristic.
The equality (\ref{lct}) is obtained by \cite{must} for characteristic $0$ and by \cite{zhu} for arbitrary characteristic.
\end{proof}

\begin{lem}\label{inequality of mld}
     Let $k$ be a perfect field of characteristic $p>0$ and $A:=\AA_k^N$, $\wa:=\AA_\CC^N$. 
     Let  $\a^e$  be a multi-ideal on $A$.
     We denote the origins in $A$ and $\wa$ by $0$ and $\tze$.
     Let $\taa^e$  be a multi-ideal  on $\wa$.
     Assume $\taa_i (\mod p)=\a_i$ for each $i$.
     Then, it follows
\begin{equation}\label{ineq-of-ideals}
\mld(0; A, \a^e)\leq \mld(\tze; \wa, \taa^e),
\end{equation}
\begin{equation}\label{ineq-of-lct}
\lct(0; A, \a^e)\leq \lct(\tze; \wa, \taa^e).
\end{equation}

\end{lem}

\begin{proof} By the assumption, there are compatible skeletons $\widetilde\Sigma \to \Sigma$  with $\widetilde\Sigma\subset \CC$ and $\Sigma\subset k$ such that the coefficients of the  generators of $\a_i$ and of $\taa_i$ are 
in $\Sigma$ and in $\widetilde\Sigma$, respectively.

 First, we prove the following inequalities which are essential for (\ref{ineq-of-ideals}) and (\ref{ineq-of-lct}):
\begin{equation}\label{ineq-of-cont}
\codim\left(\bigcap_{i=1}^r \cont^{\geq m_i}(\a_i)\cap \pi^{-1}(0), A_\infty\right)\leq 
\codim\left(\bigcap_{i=1}^r \cont^{\geq m_i}(\taa_i)\cap \pi^{-1}(\tze), \wa_\infty\right).
\end{equation}
 Let $\a_i$ be generated by $f_{ih} \in \Sigma[x_1,\ldots,x_N]$
and $\taa_i$ by their liftings $\tf_{ih} \in \wsig[x_1,\ldots,x_N]$ $(i=1,\ldots, r, h=1,\ldots, s_i)$,
then the defining ideal $I$ of $$\bigcap_{i=1}^r \cont^{\geq m_i}(\a_i)\cap \pi^{-1}(0)$$
in $A_\infty$  is generated by $x_\ell^{(0)} (\ell =1,\ldots, N)$ and
$$F_{ih}^{(j)}\in \Sigma\left[x_\ell^{(q)}\middle| \ell=1,\ldots, N, q\geq0\right]\ \ (i=1,\ldots, r, h=1,\ldots, s_i,  j\leq m_i-1).$$
Then, we can see that the liftings $x_\ell^{(0)} (\ell =1,\ldots, N)$ and
$$\widetilde F_{ih}^{(j)} \in \widetilde\Sigma\left[x_\ell^{(q)}\middle| \ell=1,\ldots, N, q\geq0\right]\ \   (i=1,\ldots, r, h=1,\ldots, s_i,  j\leq m_i-1)$$ generate the ideal 
$\widetilde I $ defining 
$$\bigcap_i \cont^{\geq m_i}(\taa_i)\cap \pi^{-1}(\tze)$$
in $\wa_\infty$.
Here, $\widetilde F_{ih}^{(j)}$ is obtained from $\widetilde f_{ih}^{(j)}$ as in Example \ref{equation}.
So, we note that
$\widetilde F_{ih}^{(j)}$ is a lifting of $F_{ih}^{(j)}$.
By Lemma \ref{ext-of-field} and  Proposition \ref{height-of-skeletons},
 it follows that $\height I \leq \height \widetilde I$,
which yields (\ref{ineq-of-cont}).
This completes the proof of (\ref{ineq-of-ideals}), by using (\ref{mld}) in Proposition \ref{formula}.
Here, we note that Proposition \ref{height-of-skeletons} states for finitely generated $k$-algebra and $\CC$-algebra,
while our statement is for infinitely generated algebras.
However as the monomials appearing in the generators and the number of generators of $I$ and $\widetilde I$ are finite,
we can work in the same way as in Proposition \ref{height-of-skeletons}.
Concretely, we may assume for $m\gg0$, the ideal $I\subset k\left[x_\ell^{(q)}\middle| \ell=1,\ldots, N, q\geq0\right]$
is generated by elements of $k\left[x_\ell^{(q)}\middle| \ell=1,\ldots, N, 0\leq q\leq m\right]$.
Let 
$$I_m\subset k\left[x_\ell^{(q)}\middle| \ell=1,\ldots, N, 0\leq q\leq m \right]$$
be the ideal generated by the generators of $I$.
In the same way, we obtain the ideal
$$\wI_m \subset \CC\left[x_\ell^{(q)}\middle| \ell=1,\ldots, N, 0\leq q\leq m \right]$$
generating $\wI$.
As $\wI_m$ is a lifting of $I_m$, applying Proposition \ref{height-of-skeletons}, we have
$$\height I_m\leq \height \wI_m,$$
which yields 
$$\height I\leq \height \wI,$$ 
where we use $\height I=\height I_m$ and $\height \wI=\height \wI_m$.
The inequality (\ref{ineq-of-ideals}) follows from this using the formula in Proposition \ref{formula}.

For (\ref{ineq-of-lct}), we need one more step.
     We note that 
   $$\lct(0;A,\a)=\inf_m \frac{\codim_0(\bigcap_i\cont^{\geq m_i}(\a_i), A_\infty)}{\sum_im_ie_i}\leq \inf_m\frac{ \codim (\bigcap_i\cont^{\geq m_i}(\a_i)\cap \pi^{-1}(0), A_\infty)}{\sum_im_ie_i},$$
   because an irreducible component of $\bigcap_i\cont^{\geq m_i}(\a_i)\cap \pi^{-1}(0)$ is contained in an
   irreducible component $T$  of 
 $\bigcap_i\cont^{\geq m_i}(\a_i)$  with $\pi(T)\cap\{0\}\neq \emptyset$.
   Then, by (\ref{ineq-of-cont}), it follows
   $$\inf _m\frac{ \codim (\bigcap_i\cont^{\geq m_i}(\a_i)\cap \pi^{-1}(0), A_\infty)}{\sum_im_ie_i}
\leq \inf_m \frac{\codim (\bigcap_i\cont^{\geq m_i}(\taa_i)\cap \pi^{-1}(\tze), \wa_\infty)}{\sum_im_ie_i}$$
  
  On the other hand, in  characteristic $0$, by \cite[Lemma 2.6]{dm}, we have the equality
    $$\lct (\tze; \wa, \taa^e)=\inf _m\frac{\codim (\bigcap_i\cont^{\geq m_i}(\taa_i)\cap \pi^{-1}(\tze), \wa_\infty)}{\sum_im_ie_i},$$
    which yields (\ref{ineq-of-lct}).
   
\end{proof}


\section{Construction of a  lifting of a sequence in positive characteristic to that in characteristic zero }

\vskip.5truecm 
\noindent
Let $k$ be an algebraically closed field of characteristic $p>0$.
In this section we construct a lifting of a given blow-up sequence by closed points over $(\AA_k^2, 0)$ to such a sequence over $\AA_\CC^2$.
%
\begin{rem}[see for example {\cite{kssw}}]\label{graded}
 Let $A_n\to A_{n-1}\to\cdots A_1\to A$ be a sequence of blow-ups with the center at  subschemes 
 over the affine space $A=\AA_k^N$.
      Then, each $A_i \ (i>0)$ is described as $A=\spec R_0$, $  A_i=\proj R_i$, where $R_i$ is a graded ring 
      with its degree 0 part $R_0=k[x_1,\ldots, x_N]$ as follows:
  \begin{equation}\label{proj} 
      R_i=\oplus _{\ell\geq0} R_0[I_1 T_1,\ldots, I_{i}T_{i}]_{(\ell g_1,\ldots,\ell g_{i-1}, \ell g_i)}, 
\end{equation}
$$    \ \ \ \mbox{for} \ \
     g_1\gg\cdots \gg g_{i}\gg ..\gg g_n,$$
   where $x_\iota$ and $ T_j$ are indeterminates for all $\iota$ and $j$, and 
   ${I}_1\subset R_0,\ldots, I_i\subset R_{i-1},\ldots, I _n\subset R_{n-1}$ are the homogeneous  ideals 
   defining the center of each blow-up.
   We define the multi-degree of an element  $r\cdot T_1^{e_1}\cdots T_i^{e_i}\in R_i$  $( r\in R_0)$ as follows:
    $$\deg r \cdot T_1^{e_1}\cdots T_i^{e_i}=(e_1,\cdots, e_i).$$  
   
   By this degree, the ring $ R_0[I_1 T_1,\ldots, I_{i}T_{i}]$ becomes a multi-graded ring 
   and $ R_0[I_1 T_1,\ldots, I_{i}T_{i}]_{(\ell g_1,\ldots,\ell g_{i-1}, \ell g_i)}$ is its homogeneous part
   of degree $(\ell g_1,\ldots,\ell g_{i-1}, \ell g_i)$.
   Then, 
   $R_i$ is a ``diagonal'' sub-graded ring of $ R_{0}[I_1 T_1,\ldots, I_{i}T_{i}]$.
   From now on, we regard $R_i$ as a graded ring with respect to $\ell\in \ZZ_{\geq0}$
   and a homogeneous ideal of $R_i$ is in this grading.
\end{rem}

\begin{rem}
   By the construction, we can see that for every $i=0,1,\ldots,n$, $R_i$ is a k-subalgebra of a polynomial ring
   $$R_i\subset k[x_1,\ldots, x_N, T_1,\ldots, T_i].$$
   We will construct a skeleton 
   $$S_i\subset \Sigma[x_1,\ldots, x_N, T_1,\ldots, T_i]$$
   of $R_i$
   for an appropriately chosen skeleton $\Sigma\subset k$. 
   By this, we can apply the statements of Lemma \ref{ext-of-field} to our situation.
   Here, note that the skeletons of this form have the following good properties:
 \begin{enumerate}
  \item
   $S_i  \hookrightarrow R_i$ is flat, because it is a base change of a flat extension $ \Sigma   \hookrightarrow k$;
   \item 
   By enlarging $\Sigma$ we can include any finite number of elements of $R_i$ into $S_i$.
  \end{enumerate}

   \end{rem}

\vskip.3truecm
\noindent
{\bf[Proof of Theorem \ref{newmain}]}
Write
\[
A_0=\spec R_0,\qquad R_0=k[x_1,\dots,x_N],
\]
and for $i\ge 1$ write $A_i=\proj R_i$ as in Remark~\ref{graded}, where
$I_i\subset R_{i-1}$ is the homogeneous ideal defining the center $P_i$ of
$\varphi_i$.

Choose a skeleton $\Sigma\subset k$ and a compatible skeleton
$\widetilde\Sigma\subset \CC$, and set
\[
S_0=\Sigma[x_1,\dots,x_N],\qquad
\widetilde S_0=\widetilde\Sigma[x_1,\dots,x_N].
\]
Let
\[
B_0:=\spec S_0,\qquad \wB_0:=\spec \widetilde S_0.
\]
More generally, once $S_i$ and $\widetilde S_i$ are constructed, we write
\[
B_i:=\proj S_i,\qquad \wB_i:=\proj \widetilde S_i.
\]
For notational simplicity, whenever $E_j\subset A_j$ denotes the exceptional
divisor of $\varphi_j$, we also denote by the same symbol $E_j\subset B_j$ the
corresponding divisor on the skeleton $B_j$. Likewise, the corresponding divisor
on $\wB_j$, and later its base change to $\wa_j$, will also be
denoted by $\we_j$.

We construct $S_i$, $\widetilde S_i$, and prime liftings $\tP_i$
inductively so that the incidence relations---that is, the relations
\[
P_i\in E_j \qquad (j<i)
\]
between the centers and the strict transform of the earlier exceptional divisors---are preserved.

For $i=1$, the center is the origin, so
\[
I_1=(x_1,\dots,x_N)\subset S_0,
\qquad
\widetilde I_1=(x_1,\dots,x_N)\subset \widetilde S_0.
\]
Let $S_1$ and $\widetilde S_1$ be obtained from $S_0$ and $\widetilde S_0$ by using $I_1$ and $\wI_1$ 
exactly as $R_1$ is obtained from $R_0$  in Remark~4.1, and set
\[
B_1:=\proj S_1,\qquad \wB_1:=\proj \widetilde S_1.
\]
Then $\wB_1\pmod p=B_1$, and $\we_1$ is a prime lifting of
$E_1$.

Now assume that $S_{i-1}$, $\widetilde S_{i-1}$, and
\[
\tP_1,\dots,\tP_{i-1}
\]
have already been constructed. Let
\[
J:=I\!\left(\bigcap_{P_i\in E_j}E_j\right)\subset R_{i-1},
\]
where $I(\bigcap_{P_i\in E_j}E_j)$ is the homogeneous ideal in $R_{i-1}$ defining $\bigcap_{P_i\in E_j}E_j$.
After enlarging $\Sigma$ if necessary, we may assume that both $J$ and $I_i$
are generated by elements of $S_{i-1}$. By the inductive hypothesis, the ideal
\[
\widetilde J:=I\!\left(\bigcap_{P_i\in E_j}\we_j\right)
\subset \widetilde S_{i-1}
\]
is a prime lifting of $J$. 
Here, as $P_i\in A_{i-1}$ is regular, the defining ideal $I_i\subset R_{i-1}$ is a regular prime ideal.
On the other hand, 
$$S_{i-1}\hookrightarrow k[S_{i-1}]=R_{i-1}$$ 
is flat, it follows that $I_i\cap S_{i-1}$ is a regular prime ideal.
We denote the point in $B_{i-1}$ corresponding to $P_i\in A_{i-1}$ by the same symbol $P_i$.
Restricting to an affine neighbourhood of $P_i$, the
maximal ideal of $P_i\in A_{i-1}$ is generated by $N$ elements 
of $(S_{i-1})_{(\ell_{i-1})}$ for some homogeneous element $\ell_{i-1}\in S_{i-1}$.
Since $P_i\in B_{i-1}$ is
a regular point, these generators form a regular system of parameters.  
Hence, we can apply Proposition \ref{lift-of-maximal-ideal}
and Corollary~\ref{inclusions} to obtain a prime lifting
\[
\tP_i\in \wB_{i-1}
\]
of $P_i$ whose defining ideal contains $\widetilde J$. Therefore
\[
\tP_i\in \we_j
\iff
P_i\in E_j
\qquad (j<i),
\]
so the incidence relations are preserved.

Let $\widetilde I_i\subset \widetilde S_{i-1}$ be the homogeneous prime ideal
defining $\tP_i$. Define $S_i$ and $\widetilde S_i$ from $S_0$ and
$\widetilde S_0$ by adjoining
\[
I_1T_1,\dots,I_iT_i
\qquad\text{and}\qquad
\widetilde I_1T_1,\dots,\widetilde I_iT_i,
\]
respectively, exactly as in Remark~\ref{graded}, and set
\[
B_i:=\proj S_i,\qquad \wB_i:=\proj \widetilde S_i.
\]
Then $\widetilde I_i$ is a prime lifting of $I_i$, and
\[
\wB_i\pmod p=B_i.
\]

Proceeding inductively, we obtain the schemes $B_i$, $\wB_i$, and the
points $\tP_i$ for all $i$. Since $\widetilde I_i$ is a prime lifting of
$I_i$, Lemma~\ref{ext-of-field}, (ii)  gives
\[
\operatorname{ht} I_i
=
\operatorname{ht} \widetilde I_i
=
\operatorname{ht}\bigl(\widetilde I_i\otimes_{\widetilde\Sigma}\CC\bigr),
\]
where the second equality uses the fact that
$\widetilde I_i\cap \widetilde\Sigma=\{0\}$.

Now set
\[
\wa_i:=\wB_i\otimes_{\widetilde\Sigma}\CC.
\]
Then $\wa_i\pmod p=A_i$ for all $i$, proving {\rm (i)}.

To prove {\rm (ii)}, fix $i$. Since $A_{i-1}$ is obtained from $\AA_k^N$ by
successive blow-ups at closed points, the point $P_i$ lies on an affine chart
isomorphic to $\AA_k^N$. After choosing suitable local coordinates on such
a chart, we may regard $P_i$ as the origin; that is, there exists an affine
neighbourhood
\[
U_{i-1}\simeq \spec k[u_1^{(i)},\dots,u_N^{(i)}]
\]
such that
\[
P_i=(u_1^{(i)},\dots,u_N^{(i)}).
\]
After enlarging $\Sigma$ finitely many times so that the coefficients involved in
this coordinate description belong to $\Sigma$, we obtain compatible affine charts
\[
U_{i-1,\Sigma}\simeq \spec \Sigma[u_1^{(i)},\dots,u_N^{(i)}],
\qquad
\widetilde U_{i-1}\simeq
\spec \widetilde\Sigma[u_1^{(i)},\dots,u_N^{(i)}],
\]
with
\[
P_i=(u_1^{(i)},\dots,u_N^{(i)}),
\qquad
\tP_i=(u_1^{(i)},\dots,u_N^{(i)}).
\]
Thus $\tP_i$ is a section over $\spec \widetilde\Sigma$. The blow-up of
$\widetilde U_{i-1}$ along this section has exceptional divisor
$\mathbb P^{N-1}_{\widetilde\Sigma}$, and after base change to $\CC$ this
becomes $\mathbb P^{N-1}_{\CC}$. In particular, the exceptional divisor is
irreducible, and the base change of
\[
\wB_i\to \wB_{i-1}
\]
to $\CC$ is the blow-up of the closed point
\[
\tP_i\in \wa_{i-1},
\]
where we regard $\tP_i=(u_1^{(i)},\dots,u_N^{(i)})\subset  \CC[u_1^{(i)},\dots,u_N^{(i)}]$
as a closed point of 
$ \wa_{i-1}$.
This proves {\rm (ii)}.

Next, by abuse of notation we denote by $\we_i\subset \wa_i$
the base change $\we_i\otimes_\wsig\CC\subset \wa_i$
of the divisor $\we_i\subset \wB_i$. 
The discrepancy of the exceptional divisor created at the $i$-th blow-up is given by
\[
k_{E_i}
=
\height I_i -1 + \sum_{P_i\in E_j} k_{E_j},
\]
and
\[
k_{\we_i}
=
\operatorname{ht}\bigl(\widetilde I_i\otimes_{\widetilde\Sigma}\CC\bigr)-1
+\sum_{\tP_i\in \we_j} k_{\we_j}.
\]
Since the heights agree and the incidence relations are preserved, induction on
$i$ yields
\[
k_{\we_i}=k_{E_i}
\qquad\text{for all }i.
\]
This proves {\rm (iii)}.

To prove (iv), let $N=2$.
For a given $f\in k[x_1,x_2]$, we may assume that $f\in S_0=\Sigma[x_1,x_2]$ by
enlarging $\Sigma$.
By the above argument we have sequences of compatible skeletons:
\[
B_n \xrightarrow{\varphi_n} B_{n-1} \to \cdots \to B_1 \xrightarrow{\varphi_1} B_0=\spec S_0,
\]
\[
\widetilde{B}_n \xrightarrow{\wv_n} \widetilde{B}_{n-1}
\to \cdots \to \widetilde{B}_1 \xrightarrow{\wv_1} \widetilde{B}_0=\spec \widetilde{S}_0,
\]
where
\[
S_0=\Sigma[x_1,x_2], \qquad \widetilde{S}_0=\widetilde{\Sigma}[x_1,x_2],
\]
and where \(E_i \subset B_i\) and \(\we_i \subset \widetilde{B}_i\) denote the
corresponding exceptional divisors. Let
\[
\varphi : B_n \to B_0, \qquad \wv : \wB_n \to \wB_0
\]
be the composites of all the blow-ups.

Let \(X \subset B_0\) be the divisor defined by \(f=0\). Write
\[
\operatorname{div}_{B_n}(f)=X_n+\sum_{i=1}^n a_iE_i,
\]
where
\[
a_i=v_{E_i}(f),
\]
and \(X_n\) contains no component \(E_i\). 
Here, by abuse of notation, we denote the strict transform of $E_i \subset B_i$ on $B_n$ by $E_i$, and that of $\we_i \subset \wB_i$ on $\wB_n$ by $\we_i$.

Denote the total transforms of
\[
(\we_i \subset \wB_i) \quad \text{and} \quad (E_i \subset B_i)
\]
by
\[
(\we_i^{\,*} \subset \widetilde{B}_n) \quad \text{and} \quad (E_i^{*} \subset B_n),
\]
respectively. We also have the expression:
\[
\operatorname{div}_{B_n}(f)=X_n+\sum_{i=1}^n b_iE^*_i,
\]
with
\[
b_i=\mult_{P_i} X_{i-1},
\]
where $X_{i-1}\subset B_{i-1}$ is the strict transform of $X\subset B_0$.

Set
\[
\L:=\mathcal{O}_{B_n}\!\left(-\sum_{i=1}^n b_iE^*_i\right), \qquad
\wL:=\mathcal{O}_{\wB_n}\!\left(-\sum_{i=1}^n b_i\we^*_i\right).
\]
%
Since $\we_i\subset \wB_i$ is a lifting of $E_i\subset B_i$,
each \(\we^*_i\subset \wB_n\) is a lifting of \(E^*_i\subset B_n\), we have
\[
\wL\otimes_{\mathbb{Z}}\mathbb{Z}/(p)\cong \L.
\]

By construction, \(\L\) is nef on \(B_n \otimes_{\Sigma} Q(\Sigma)\). On the other
hand,  the combinatorial structures (weighted dual graphs) of
 \(\{E_i\}\) and \(\{\we_i\}\)
coincide. Hence, the first equality in the following induces the second equality: 
\[
\sum_{i=1}^n a_iE_i=\sum_{i=1}^n b_iE_i^{*},
\qquad
\sum_{i=1}^n a_i\we_i=\sum_{i=1}^n b_i\we_i^{\,*}.
\]

Since the fibers of
\[
B_n \otimes_{\Sigma} Q(\Sigma) \to B_0 \otimes_{\Sigma} Q(\Sigma)
\]
have dimension at most one, and since
\[
R^1\varphi_*\mathcal{O}_{B_n \otimes_{\Sigma} Q(\Sigma)}=0,
\]
we can apply Lipman's vanishing theorem (\cite[Theorem 12.1, (ii)]{l}) to obtain
\[
R^1\varphi_*\!\left(\L \otimes_{\Sigma} Q(\Sigma)\right)=0,
\]
Since \(\Sigma \hookrightarrow Q(\Sigma)\) and
\(\widetilde{\Sigma} \hookrightarrow Q(\widetilde{\Sigma})\) are flat, Hartshorne
III, Proposition 9.3 gives
\[
(R^1\varphi_*\L)\otimes_{\Sigma} Q(\Sigma)=0,
\]

Now let
\[
\pi : B_0=\spec S_0 \to \spec \Sigma,
\qquad
\widetilde{\pi} : \widetilde{B}_0=\spec \widetilde{S}_0 \to \spec \widetilde{\Sigma}
\]
be the structure morphisms. 
Then, 
$$
\pi_*((R^1\varphi_*\L)\otimes_{\Sigma} Q(\Sigma))=0.
$$
Applying the above argument once more and Leray spectral sequence,  we obtain
\begin{equation}\label{local-vanish}
(R^1(\pi \circ \varphi)_*\L)\otimes_{\Sigma} Q(\Sigma)=0.
\end{equation}

We next prove the following claim.

\medskip
\noindent
{\bf Claim.}
\(R^1(\widetilde{\pi}\circ\wv)_*\wL=0\) in a neighborhood of the point
\(p_{\widetilde{\Sigma}} \in \spec \widetilde{\Sigma}\) lying over \(p \in \spec \mathbb{Z}\).

\medskip
\noindent
Assume the claim for the moment. Applying \((\widetilde{\pi}\circ\wv)_*\) to
\[
0 \to \wL \xrightarrow{\times p} \wL \to \L \to 0,
\]
the vanishing in the claim yields a surjection
\[
(\widetilde{\pi}\circ\wv)_*\wL
\longrightarrow
(\pi\circ\varphi)_*\L.
\]
Since \(f \in (\pi\circ\varphi)_*\L \subset \Sigma[x_1,x_2]\), we may choose a lift
\[
\widetilde{f} \in (\widetilde{\pi}\circ\wv)_*\wL=(\widetilde{\pi}\circ\wv)_*\oo_{\wB_n}\left(-\sum a_i\we_i\right)
\subset \widetilde{\Sigma}[x_1,x_2].
\]
Then
\[
v_{\we_i}(\widetilde{f}) \ge a_i
\qquad (i=1,\dots,n).
\]
Conversely, since each \(\we_i\) is a prime lifting of \(E_i\), reduction
modulo \(p\) gives
\[
v_{E_i}(f) \ge v_{\we_i}(\widetilde{f})
\qquad (i=1,\dots,n).
\]
Because \(a_i=v_{E_i}(f)\), we conclude that
\[
v_{E_i}(f)=v_{\we_i}(\widetilde{f})
\qquad (i=1,\dots,n).
\]

It remains to prove the claim. The long exact sequence associated to
\[
0 \to \wL \xrightarrow{\times p} \wL \to L \to 0
\]
contains
\[
\to R^1(\widetilde{\pi}\circ\wv)_*\wL
\xrightarrow{\times p}
R^1(\widetilde{\pi}\circ\wv)_*\wL
\to R^1(\pi\circ\varphi)_*\L \to \cdots .
\]
Hence we obtain an inclusion from the cokernel of the map  $\times p$
\[
\bigl(R^1(\widetilde{\pi}\circ\wv)_*\wL\bigr)
\otimes_{\widetilde{\Sigma}} \Sigma
\hookrightarrow
R^1(\pi\circ\varphi)_*\L.
\]
Tensoring with \(Q(\Sigma)\) over \(\Sigma\), and using the flatness of
\(\Sigma \hookrightarrow Q(\Sigma)\), we get
\[
\bigl(R^1(\widetilde{\pi}\circ\wv)_*\wL\bigr)
\otimes_{\widetilde{\Sigma}} Q(\Sigma)
\hookrightarrow
\bigl(R^1(\pi\circ\varphi)_*\L\bigr)\otimes_{\Sigma} Q(\Sigma).
\]
Since the right-hand side vanishes by (\ref{local-vanish}), we obtain
\[
\bigl(R^1(\widetilde{\pi}\circ\wv)_*\wL\bigr)
\otimes_{\widetilde{\Sigma}} Q(\Sigma)=0.
\]

To see that \(R^1(\widetilde{\pi}\circ\wv)_*\wL\) is coherent on
\(\spec \widetilde{\Sigma}\), note first that
\[
\wv : \widetilde{B}_n \to \widetilde{B}_0
\]
is projective, so \(R^1\wv_*\wL\) is coherent on
\(\widetilde{B}_0=\spec \widetilde{S}_0\). Its support is contained in
\(Z(x_1,x_2)\), hence for some \(m \gg 0\) it is annihilated by \((x_1,x_2)^m\).
Therefore it may be regarded as a coherent sheaf on
\[
\spec \widetilde{S}_0/(x_1,x_2)^m.
\]
Let
\[
\rho : \spec \widetilde{S}_0/(x_1,x_2)^m \to \spec \widetilde{\Sigma}
\]
be the induced finite morphism. Then
\[
\rho_*\bigl(R^1\wv_*\wL\bigr)
\]
is coherent on \(\spec \widetilde{\Sigma}\), and the Leray spectral sequence identifies
it with
\[
R^1(\widetilde{\pi}\circ\wv)_*\wL.
\]
Thus \(R^1(\widetilde{\pi}\circ\wv)_*\wL\) is indeed coherent on
\(\spec \widetilde{\Sigma}\).

Since its localization at \(p_{\widetilde{\Sigma}}\) becomes zero after tensoring with the
fraction field, Nakayama's lemma implies that
\[
R^1(\widetilde{\pi}\circ\wv)_*\wL=0
\]
in a neighborhood of \(p_{\widetilde{\Sigma}}\). This proves the claim, and hence the equality in (iv) of the
theorem. $\Box$



      Here, we note that the point $\tP_i\in \wa_{i-1}$ is smooth,
       as the local ring $\oo_{\wa_{i-1},\tP_i }$ has the maximal ideal generated by $N$ elements, where
       $N$ is the height of the defining ideal of $P_i$ which is equal to that of $\tP_i$.
       
       As is described in the general situation of Lemma \ref{ext-of-field}, (iv), the extension of $\a_K$ for a prime ideal $\a$ in the skeleton $S$ 
       is not necessarily prime, {\it i.e.,} geometrically, $\we_i\otimes \CC$ is not necessarily irreducible.
       However, in the special case where each blow-up center is a section over $\spec \wsig$,  the extension of an exceptional divisor is unique.
     

 
\begin{rem}
Under the notation of Theorem \ref{newmain}, the reader may expect that for $N\geq 3$ and an element $f\in k[x_1,\ldots, x_N]$,
there exists a lifting $\tf\in \CC[x_1,\ldots, x_N]$ such that 
\begin{equation}\label{=holds}
v_{E_i}(f)=v_{\we_i}(\tf).
\end{equation}

In some simple cases, namely when $i = 1,2$, we can verify this statement (see Proposition~\ref{baby} below).
However, for $N \geq 3$ and $i \geq 3$, there exists a counterexample constructed by Koll\'ar (cf. \cite{ko}).

For $N \geq 3$, the condition~(\ref{=holds}) appears to be too strong for our purposes,
since the counterexample does not affect the set of log canonical thresholds ($\lct$s).
This observation suggests that the statement should be reformulated in the case $N \geq 3$.
This will be studied in our forthcoming paper.

\end{rem}

\begin{prop}\label{baby}

Assume the notation of Theorem~1.1, and set $A:=\mathbb{A}_k^N$.

\begin{enumerate}
\item
Let $\varphi_1 \colon A_1 \to A$ be the blow-up of the origin $0 \in A$, and
let $E_1 \subset A_1$ be its exceptional divisor.
Then, for every nonzero polynomial $f\in  k[x_1,\dots,x_N]$,
there exist a lifting $\we_1$ of $E_1$ and a lifting
$\tf\in\CC[x_1,\dots,x_N]$
of $f$ such that
\[
v_{E_1}(f) = v_{\we_1}(\tf).
\]

\item
Let $\varphi_2 \colon A_2 \to A_1$ be the blow-up of a $k$-valued closed point
\[
P_2 \in E_1 \subset A_1,
\]
and let $E_2 \subset A_2$ be its exceptional divisor.
Then, for every nonzero polynomial $f\in  k[x_1,\dots,x_N]$,
there exist a lifting $\we_2$ of $E_2$ and a lifting
$\tf\in\CC[x_1,\dots,x_N]$
of $f$ such that
\[
v_{E_i}(f) = v_{\we_i}(\tf)\qquad \mbox{for}\qquad i=1,2
\]
\end{enumerate}

\end{prop}

\begin{proof} For (i), take appropriate compatible skeletons $\Sigma$ and $\wsig$ so that $f\in \Sigma[x_1,\ldots, x_N]$.
For  $$f=\sum_{i\in \ZZ^N_{\geq 0}} a_i{\bold{x}}^i\ \ \  ,\ \ \mbox{take}\ \ \ \tf=\sum_{i\in \ZZ^N_{\geq 0}} \ta_i\bold{x}^i,$$
where $\ta_i\in \wsig$ is a lifting of $a_i\in \Sigma$ such that 
 $\ta_i=0$ if $ a_i=0$. 
Then, we can see that $\ord_{\bold x} f=\ord_{\bold x} \tf$, which implies 
$$v_{E_1}(f)=v_{\we_1}(\tf).$$

For (ii),  we may  assume that locally the following hold as $P_2$ is a $k$-valued point:
$$(B_{1}, P_{2})=\left(\spec \Sigma \left[x_1,\frac{x_2}{x_1},\ldots, \frac{x_N}{x_1}\right], \left(x_1,\frac{x_2}{x_1},\ldots, \frac{x_N}{x_1}\right)\right)$$
and
  $$(\wB_{1}, \tP_{2})=\left(\spec \wsig \left[x_1,\frac{x_2}{x_1},\ldots, \frac{x_N}{x_1}\right], \left(x_1,\frac{x_2}{x_1},\ldots, \frac{x_N}{x_1}\right)\right),$$
  for an appropriate compatible skeletons $\Sigma\subset k$ and $\wsig\subset\CC$.
 
  For $f\in \Sigma[x_1,\ldots, x_N]$, 
    if $\ord_{\bold x}f=m_1$, then $f$ factors as  follows:
  $$f=x_1^{m_1} g \left(x_1,\frac{x_2}{x_1},\ldots, \frac{x_N}{x_1}\right)\ \mbox{with}\ \  x_1\not| g$$
  on $B_{1}$.
  If $$g=\sum_e a_e x_1^{e_1}\left( \frac{x_2}{x_1}\right)^{e_2}\ldots  \left(\frac{x_N}{x_1}\right)^{e_N},$$  where  $a_e\in \Sigma$
  and $e=(e_1,\ldots, e_N)$,
  define $$\tg =\sum_e \ta_e x_1^{e_1}\left( \frac{x_2}{x_1}\right)^{e_2}\ldots  \left(\frac{x_N}{x_1}\right)^{e_N},$$
  by using a lifting $\ta_e\in \wsig$ of $a_e$ 
  such that $\ta_e=0$ if $a_e=0$.
  Then, $\mult_{P_{2}}g=\mult_{\tP_{2}}\tg$, which we denote by $m_2$.
  Therefore, the lifting
 $\tf:=x_1^{m_1}\tg$ of $f$ satisfies
  $$v_{E_{2}}(f)=m_1+m_2=v_{\we_{2}}(\tf),$$
  which yields (ii).
  
\end{proof}

\section{Proofs of Applications}


\medskip
\noindent
{\bf [Proof of Corollary \ref{main}]}

\noindent 
 Let $A_n\to A_{n-1}\to\cdots \to A_1\to A_0=\AA_k^2$ be the sequence of blow-ups yielding the given prime divisor $E$
 as the last exceptional divisor $E_n$.

Since the ideal $\a \subset k[x_1,x_2]$ has finitely many generators and only finitely many coefficients in $k$, by enlarging $\Sigma$ we may assume that $S_0$ contains all these generators. 
For all generators, there exist liftings satisfying Theorem~1.1(ii). Let these lifted polynomials generate an ideal
\[
\taa\subset \CC[x_1,x_2].
\]
Then
\[
v_E(\a)=v_{\we}(\taa),
\]
where $\we$ is a lifting of $E=E_n$ constructed in Theorem~1.1. By this and Theorem~1.1(i), assertions (i) and (ii) of Corollary~\ref{main} follow. For (iii), it suffices to include all coefficients of the generators of $\a_1,\ldots,\a_r$ in $\Sigma$. Then, for every $i$, we have
\[
v_E(\a_i)=v_{\we}(\taa_i),
\]
which proves (iii).
  $\Box$


\vskip.3truecm
\noindent
Until now, we assume the base space $A$ to be the affine space.
 However, 
we can apply the theorem also to a pair consisting of  a smooth  variety and a multi-ideal with a real exponent.
In such a case, we cannot define a lifting of a regular function around the point $0$,
but passing through the completion we can reduce the discussion to the case of the affine space as follows:

\vskip.3truecm
\noindent
    {\bf[Proof of Corollary \ref{equal-p-0}]}

\noindent
 Let $\widehat{A}:=\spec \widehat\oo_{A,0}$ and $\widehat0\in \widehat{A}$ the closed point.
    Then, the set of prime divisors over $(A,0)$ coincides with the set of  prime divisors over $(\widehat{A},\widehat0)$.
    On the other hand,  the ideals $\a_i\subset \oo_A$ $(i=1,\ldots,r)$ give the extensions    $\widehat\a_i\subset \widehat\oo_{A,0}=k[[x_1,x_2]]$ and satisfy the following:
    $$v_E(\a^e)=v_E(\widehat\a^e)\ \ \ \mbox{for\ any \ prime\ divisor\ } E \  \mbox{over} \ (A,0)\ \mbox{and}$$ 
\begin{equation}\label{lct-of-ext-1}    
    \lct (0;A, \a^e)=\lct (\widehat0; \widehat{A}, \widehat\a^e).
\end{equation}    
    Then, for a given prime divisor $E$ over $(A,0)$ there exists $d\in\NN$ such that 
    $v_E(\widehat\a_i)<v_E(\widehat\m^d)$ for all $i$, where $\widehat\m$ is the maximal ideal of $\widehat\oo_{A,0}
    =k[[x_1,x_2]]$.
    Here, noting that $\widehat\a_i+\widehat\m^d$ is the extension of an ideal $ \underline{\a}_{id}\subset k[x_1,x_2]$,
    we obtain
    $$v_E(\widehat\a^e)=v_E((\widehat\a+\widehat\m^d)^e)=v_E(\underline{\a}_{d}^e) \ \mbox{and}$$
 \begin{equation}\label{lct-of-ext}   
    \lct(0;\AA_k^N, \underline{\a}_{d}^e) \geq \lct (\widehat0; \widehat{A}, \widehat\a^e),
\end{equation}    
    where the notation $(\widehat\a+\widehat\m^d)^e$ means the formal product
    $$\prod_{i=1}^r(\widehat\a_i+\widehat\m^d)^{e_i}.$$
    Here, 
    the above inequality (\ref{lct-of-ext}) of lcts follows from $\underline{\a}_{id}k[[x_1,\ldots, x_N]]= \widehat\a_i+\widehat\m^d 
    \supset \widehat\a_i$.
    Now, by (\ref{lct-of-ext}) we have $(\AA_k^2, \underline\a_d^e)$ is log canonical, and for $E$ we have
     $a(E; A, \a^e)=a(E;\widehat{A}, (\widehat\a+\widehat\m^d)^e)=a(E; \AA_k^N, \underline{\a}_{d}^e)$.
     Therefore the log discrepancy   $a(E; A, \a^e)$   belongs to the set
 $$\Lambda_{(\AA_k^N,0),e}:= \left\{a(E;\AA_k^2,\a^e)  \ \middle | \begin{array}{l}\ (\AA_k^2, \a^e)\ \mbox{is\ log\ canonical\ at $0$}, 
         \\ E:\ \mbox{\ any \ prime divisor\ over } (\AA_k^2,0) \end{array}\right\} $$     
     Now, we can  reduce  the discussion to the case
    over $(\AA_k^N, 0)$.
      By  Corollary \ref{main}, (iii), for a fixed exponent $e\in \RR_{>0}^r$, 
     we have the inclusion:
     $$\Lambda_{(\AA_k^2,0),e}\subset \Lambda_{(\AA_\CC^2,0),e},$$
     where, we use that a lifting $\taa^e$ of log canonical multi-ideal $\a^e$ is also log canonical
     by Lemma \ref{inequality of mld}.
     As the right set is discrete in the sense of Corollary \ref{equal-p-0} due to \cite{kawk2}, the left set  $\Lambda_{(\AA_k^N,0),e}$ is also discrete.
     $\Box$   
  
 \vskip.3truecm
 \noindent
 {\bf[Proof of Corollary \ref{mld-contain}]}
 
 \noindent
 For containment of mlds, it is sufficient to show 
  for any multi-ideal $\a^e$ on a smooth surface $A$ defined over $k$, there exists a multi-ideal $\taa^e$ on $\AA_\CC^2$ such that 
    $$\mld(0;A, \a^e)=\mld(0; \AA_\CC^2,\taa^e).$$
  
      As in the proof of Corollary \ref{equal-p-0}, we may assume that 
$(A,0)=(\AA_k^2,0)$.
Let $E$ be a prime divisor over $(\AA_k^2,0)$ computing 
$\mld(0;\AA_k^2,\a^e)$.
Then, by Corollary \ref{main}, there exist  liftings $\taa^e$ of $\a^e$
and  $\we$ of $E$ such that
\[
a(E;\AA_k^2,\a^e)=a(\we;\AA_{\CC}^2,\taa^e),
\]
which yields the following inequality:
\begin{equation}\label{computing}
\mld(0;\AA_k^2,\a^e)
= a(E;\AA_k^2,\a^e)
= a(\we;\AA_{\CC}^2,\taa^e)
\ge \mld(0;\AA_{\CC}^2,\taa^e).
\end{equation}

On the other hand, by Lemma \ref{inequality of mld}, we obtain
\[
\mld(0;\AA_k^2,\a^e)\le \mld(0;\AA_{\CC}^2,\taa^e),
\]
which implies that the inequality in (\ref{computing}) is in fact an equality. 

For containment of lcts, first in case that $\lct(0; \AA_k^2, \a^e)$ is computed by a prime divisor over $(\AA_k^2, 0)$,
we can prove that 
$$\lct(0; \AA_k^2, \a^e)=\lct(0;\AA_\CC^2,\taa^e),$$
for an appropriate lifting $\taa^e$ of $\a^e$, in the same way as for mlds.
If $c:=\lct(0; \AA_k^2, \a^e)$ is not computed by a prime divisor over $(\AA_k^2, 0)$, 
then it is computed by a curve on $\AA_k^2$.
Then, $$c=\lct(0;\AA_k^1, \b^e)$$
for some multi-ideal $\b^e$ on $\AA_k^1$.
It is known that the set of lcts of multi-ideal with the exponent $e$ on a smooth curve
does not depend on the choice of the base field.
Therefore, $$c=\lct(0;\AA_\CC^1, \tb^e)$$
for some multi-ideal $\tb^e$ on $\AA_\CC^1$.
By \cite{dm}, $c$ becomes also the lct of a multi-ideal with the exponent $e$ on $\AA_\CC^2$.
%
%
$\Box$

\vskip.3truecm
\noindent
{\bf[Proof of Corollary \ref{campillo}]} 
  Let 
$$A_n\stackrel{\varphi_n}\longrightarrow A_{n-1} \to
       \cdots A_1 \stackrel{\varphi_1}\longrightarrow A$$
  be the sequence of blow-ups at closed points used to obtain an embedded resolution of $(\AA_k^2, C)$ around $0\in \AA_k^2$.
Let  $$\wa_n\stackrel{\widetilde\varphi_n}\longrightarrow \wa_{n-1} \to
       \cdots \wa_1 \stackrel{\widetilde\varphi_1}\longrightarrow \wa=\AA_\CC^2,$$
be its lifting constructed in the proof of Theorem \ref{newmain}.
 Then, by construction, the proximity relations among the exceptional curves over $\AA_k^2$ and those over $\AA_\CC^2$ coincide.
Moreover, the corresponding multiplicity systems coincide at each step of the resolution process.
$\Box$

\vskip1truecm

\noindent Shihoko Ishii, \\ Graduate School of Mathematical Science, University of Tokyo,\\
3-8-1, Komaba. Meguro, Tokyo, Japan.\\
shihoko@g.ecc.u-tokyo.ac.jp


\begin{thebibliography}{11}





\bibitem{cam} A. Campillo, {\em Algebroid Curves in Positive Characteristics}, Lecture Notes in Math. {\bf 813}, Springer, (1980).



\bibitem{dm} T. De Fernex and M.Musta\cedilla{t}\v{a}, {\em Limit of log canonical thresholds},
Ann. scientifiques de l'\'Ecole Normale Sup\'erieure, S\'erie 4, {\bf 42} (2009) no. 3, 491--515.

\bibitem{EM} L. Ein and M.Musta\cedilla{t}\v{a}, {\em Jet schemes and singularities}, Proc. Symp. Pure
Math. {\bf 80}. 2, (2009) 505--546.


\bibitem{gr} G-M. Greuel, {\em Singularities in Positive Characteristic,} arXiv:1711.03453 (2017).


\bibitem{ha} R. Hartshorne, {\em Algebraic Geometry,} GTM, {\bf 52}, Springer-Verlag,
(1977) 496 pages.
 
\bibitem{ir2} S. Ishii and A. Reguera, {\em Singularities  in arbitrary characteristic 
via jet schemes},   Hodge theory and $L^2$ analysis (2017)
419--449.



\bibitem{i2} S. Ishii, {\em The minimal log discrepancies on a smooth surface in positive    
                       characteristic}, Math. Zeitschrift, {\bf 297}, (2021) 389--397.

\bibitem{inv} S. Ishii, {\em  Inversion of modulo $p$ reduction and a partial descent from characteristic $0$ to positive characteristic},  Romanian J.  Pure and Applied Math.
Vol. LXIV, No.4 (2019) 431--459.

\bibitem{i3} S. Ishii, {\em Introduction to Singularities, 2nd ed}, Springer-Verlag, (2018) 236 pages.



\bibitem{kawk2} M. Kawakita, {\em Discreteness of log discrepancies over log canonical triples on a fixed pair},  J. Alg. Geom. {\bf  23}, (4), (2014) 765--774.





\bibitem{ko} J. Koll\'ar,  {\em Planar, rational curves over $\Bbb{F}_2$ whose only singularity is a double point}, arXiv:2603.04211 (2026). 


\bibitem{kssw} K. Kurano, Ei-ichi Sato, A. K. Singh and K-i. Watanabe, 
{\em Multigraded rings, diagonal subalgebras, and rational singularities}, 
J. Alg. 
{\bf 322},(9),(2009) 3248--3267.


\bibitem{l} J. Lipman, {\em Rational singularities with applications to algebraic
surfaces and unique factorization},
Publ. math.I.H.\'E.S., tome 36 (1969), p. 195--279.





\bibitem{mats} H. Matsumura, {\em Commutative Ring Theory}, Cambridge St. Ad. Math. 
{\bf 8}, Cambridge UP. (1980) 320 pages. 

\bibitem{must} M. Musta\cedilla{t}\v{a}, {\em Singularities of Pairs via Jet Schemes,}
           J. Amer. Math. Soc. 15 (2002), 599--615.




\bibitem{stack} The Stacks Project Authors, {\em Stacks Project}, {https://stacks.math.columbia.edu/tag/0C50}

\bibitem{tao} T. Tao, {\em Rectification and the Lefschetz principle},
                 https://terrytao.wordpress.com/tag/lefschetz-principle/ .


\bibitem{zhu} Z. Zhu, {\em Log canonical thresholds in positive characteristic}, 
Math. Zeit. {\bf 287},  (2017)1235--1253.





\end{thebibliography}
\end{document}